\documentclass{amsart}
\usepackage{enumerate}
\usepackage{amssymb}
\usepackage{mathtools}
\usepackage{amsmath}
\usepackage{color}
\usepackage[
colorlinks=true,
linkcolor=blue,
citecolor=blue,
urlcolor=blue]{hyperref}

\numberwithin{equation}{section}

\newcommand{\normmm}[1]{{\left\vert\kern-0.25ex\left\vert\kern-0.25ex\left\vert #1
\right\vert\kern-0.25ex\right\vert\kern-0.25ex\right\vert}}

\newtheorem{thm}{Theorem}[section]
\newtheorem{lem}[thm]{Lemma}

\newtheorem{rem}[thm]{Remark}

\newtheorem{prop}[thm]{Proposition}

\newtheorem{conj}[thm]{Conjecture}

\newcommand{\N}{{\mathbb N}}

\newcommand{\Z}{{\mathbb{Z}^{n}_{m}}}

\begin{document}

\title[Strong Convolution Inequality]{The Strong Naor--Schechtman Convolution Inequality on the Product of Cyclic Groups}

\author[Jiao]{Yong Jiao}
\address{School of Mathematics and Statistics, HNP-LAMA, Central South University, Changsha 410075, China}
\email{jiaoyong@csu.edu.cn}

\author[Luo]{Sijie Luo}
\address{School of Mathematics and Statistics, HNP-LAMA, Central South University, Changsha 410075, China}
\email{sijieluo@csu.edu.cn}

\author[Zanin]{Dmitriy Zanin}
\address{School of Mathematics and Statistics, HNP-LAMA, Central South University, Changsha 410075, China}
\email{d.zanin@unsw.edu.au}

\author[Zhou]{Dejian Zhou}
\address{School of Mathematics and Statistics, HNP-LAMA, Central South University, Changsha 410075, China}
\email{zhoudejian@csu.edu.cn}
\date{}

\begin{abstract}
In this paper, we resolve the convolution inequality problem posed by Naor and Schechtman in \cite[Question 6.1]{N-S2016} (see also \cite{Na2016}). More precisely, we establish the following strong convolution inequality on products of finite cyclic groups: For each $1\leq p<\infty$ there exists $c_{p}>0$ such that for every $n$, $m\in \mathbb{N}$ the following holds
\[
\begin{split}
&\sum_{\varepsilon\in \{-1,1\}^{n}}\sum_{x\in \Z}\left|E_{\{1,\cdots,n\}}f(x+\varepsilon)-E_{\{1,\cdots,n\}}f(x-\varepsilon)\right|^{p}\\
\leq& c^{p}_{p}\sum_{\varepsilon\in \{-1,1\}^{n}}\sum_{x\in \Z}\left|\varepsilon_{j}\left[E_{\{1,\cdots,n\}\setminus\{j\}}f(x+e_{j})-E_{\{1,\cdots,n\}\setminus\{j\}}f(x-e_{j})\right]\right|^{p},
\end{split}
\]
$E_{A}f(x)=\frac{1}{2^{n}}\sum_{\delta\in \{-1,1\}^{n}}f\left(x+\sum_{j\in A}\delta_{j}e_{j}\right)$ for every $A\subseteq \{1,\cdots,n\}$. The growth order we obtain here $c_{p}=O(p)$ as $p\to \infty$. For the endpoint $p=\infty$, we show by an counterexample that the strong convolution inequality fails for the $L_{\infty}$, and we prove that the appropriate endpoint for $p=\infty$ is the martingale $BMO$ space.
\end{abstract}

\maketitle

\section{Introduction}

The study of metric property determined only by the distance of Banach spaces forms a central theme in the Ribe program. The metric $X_{p}$ inequality introduced by Naor and Schechtman \cite{N-S2016} is a nonlinear counterpart of the Rosenthal inequality, which serves as an obstruction to certify the failure of bi-Lipschitz emebddability of $L_{q}$ into $L_{p}$ for $2<q<p<\infty$.

To make statements precisely, we introduce necessary notations that will be used throughly. For a given $n$, let $\Z$ be the $n$-folds Cartesian product of the cyclic group $\mathbb{Z}_{m}$. We denote that $\Omega_{n}=\{-1,1\}^{n}$ the $n$-dimensional hypercube. Both $\Z$ and $\Omega_{n}$ are equipped with the uniform probability measures $\nu$ and $\mu$, respectively. Let $[n]=\{1,\cdots, n\}$, and for each $j\in [n]$ we define the translation operator $T_{j}$ by
\[
T_{j}f(x)\coloneqq f(x+e_{j}),\qquad f:\Z\to \mathbb{R},
\]
and $e_{j}=(0,\cdots, 1,\cdots,0)$ is the vector with $1$ appears in the $j$-th position. It is clear that the linear operator $T_{j}$ is an isometry on $L_{p}(\Z)$ for all $p\in [1,\infty]$ with its inverse given by
\[
T^{-1}_{j}f(x)=f(x-e_{j}),\qquad f:\Z\to \mathbb{R}.
\]
The $j$-th expectation operator $E_{j}$ and the $j$-th derivative operator $D_{j}$ are defined as follows
\[
E_{j}\coloneqq \frac{T_{j}+T^{-1}_{j}}{2}\quad\mbox{and}\quad D_{j}=\frac{T_{j}-T^{-1}_{j}}{2}.
\]
It is clear that the family $\{E_{j}\}_{j=1}^{n}$ commute with each other. For each $A\subseteq [n]$, let $E_{A}\coloneqq \prod_{j\in A}E_{j}$ and $E_{\emptyset}=\mathrm{Id}$. Hence,
\[
E_{A}f(x)=\frac{1}{2^{|A|}}\sum_{\delta\in \{-1,1\}^{A}}f\left(x+\sum_{j\in A}\delta_{j}e_{j}\right)=\frac{1}{2^{n}}\sum_{\delta\in \Omega_{n}}f\left(x+\sum_{j\in A}\delta_{j}e_{j}\right),
\]
for every $f:\Z\to \mathbb{R}$, where $|A|$ stands for the cardinality of $A$. For $A\subseteq [n]$ and $\varepsilon\in \Omega_{n}$, we let
\[
T^{\varepsilon}_{A}\coloneqq \prod_{j\in A}T^{\varepsilon_{j}}_{j}.
\]

The metric $X_{p}$ inequality of Naor and Schechtman is stated as follows, which is one of the fundamental results in \cite{N-S2016}.
\begin{thm}[Naor--Schechtman]
For $2\leq p<\infty$, and $k$, $m\in \mathbb{N}$ with $k\in [n]$ and
\[
m\geq \frac{n^{3/2}\log(p)}{\sqrt{k}}+pn,
\]
then there exists a constant $c_{p}>0$ depending only on $p$ such that the following inequality holds for every $f:\mathbb{Z}^{n}_{4m}\to L_{p}$
\[
\begin{split}
&\frac{(p\log(p))^{-1}}{\binom{n}{k}}\sum_{\substack{S\subseteq[n] \\|S|=k}}\frac{\mathbb{E}\left[\left\|f(x+2m\varepsilon_{S})-f(x)\right\|^{p}_{L_{p}}\right]}{m^{p}}\\
\leq& c_{p}\left(\frac{k}{n}\sum_{j=1}^{n}\mathbb{E}\left\|T_{j}f(x)-f(x)\right\|^{p}_{L_{p}}+\left(\frac{k}{n}\right)^{p/2}\mathbb{E}\left\|T^{\varepsilon}_{[n]}f(x)-f(x)\right\|^{p}_{L_{p}}\right),
\end{split}
\]
where the expectation is with respect to $(x,\varepsilon)\in \mathbb{Z}^{n}_{4m}\times \{-1,1\}^{n}$ chosen unifromly at random.
\end{thm}
\noindent In the same paper, Naor and Schechtman showed that the scaling appears in the metric $X_{p}$ inequality is necessary (see \cite[Propostion 1.4]{N-S2016}). This leads them to conjecture that $L_{p}$ space is an $X_{p}$ metric space with optimal scaling parameter.
\begin{conj}[Naor--Schechtman \cite{N-S2016}]\label{N-S conjecture}
For every $2<p<\infty$ there exist positive constants $\alpha_{p}$ and $c_{p}$ such that for every $n$, $m\in \mathbb{N}$ with $k\in[n]$, $m\geq C_{p}\sqrt{n/k}$, then for every $f:\mathbb{Z}^{n}_{4m}\to \mathbb{R}$ we have
\[
\begin{split}
&\frac{\alpha_{p}}{\binom{n}{k}}\sum_{\substack{S\subseteq [n] \\|S|=k}}\frac{\mathbb{E}\left[\left|f(x+2m\varepsilon_{S})-f(x)\right|^{p}\right]}{m^{p}}\\
\leq& c_{p}\left(\frac{k}{n}\sum_{j=1}^{n}\left\|T_{j}f-f\right\|_{L_{p}(\Z\times \Omega_{n})}+\left(\frac{k}{n}\right)^{p/2}\mathbb{E}\left\|T^{\varepsilon}_{[n]}f-f\right\|^{p}_{L_{p}(\Z\times \Omega_{n})}\right),
\end{split}
\]
where the expectation is with respect to $(x,\varepsilon)\in \mathbb{Z}^{n}_{4m}\times \{-1,1\}^{n}$ chosen unifromly at random.
\end{conj}

In \cite[Section 6]{N-S2016}, Naor and Schechtman proposed an approach to Conjecture \ref{N-S conjecture} based on a convolution inequality, showing that the inequality formulated in \cite[Question 6.1]{N-S2016} would imply the conjecture. This convolution inequality therefore provides an analytic tool to the optimal scaling and its consequences for the nonlinear geometry of $L_{p}$ spaces. Subsequently, Naor \cite{Na2016} resolved the scaling conjecture via a different approach. He used Lust-Piquard's bounds for discrete Riesz transforms \cite{LP1998} to established the sharp metric $X_{p}$ inequality, while explicitly noting that his argument left Question $6.1$ unresolved. More recently, Areshidze \cite{Are2026} obtained the optimal $p/\log p$ bounds for higher-order Rademacher chaos via a martingale argument. Related developments have also been made in the study of discrete Riesz transform estimates. Domelevo, Ivanisvili, Petermichl, and Volberg \cite{D-I-P-V2026} gave a commutative Bellman-function proof of the discrete Riesz transform bound for $2\leq p<\infty$, with a constant bounded by $C(p-1)$, where $C$ is universal. In the complementary range $1<p<2$, Jiao, Luo, Zanin, and Zhou \cite{JLZZ2026} introduced an exponential lifting technique and
combined it with martingale methods to establish the sharp fractional Riesz estimate, while Xu and Zhang \cite{X-Z2026} independently established the same estimate using noncommutative BMO theory, thereby resolving the problem posed by Ben-Efraim and Lust-Piquard \cite{ELP2008}.

Before stating the conjecture precisely, we introduce two basic mappings to make the statement neater. For $j\in [n]$ and every $f:\Z\to \mathbb{R}$ and $j\in [n]$, we define functions $H_{f}: \Z\times \Omega_{n}\to \mathbb{R}$ by
\[
H_{f}(x,\varepsilon)=(E_{[n]} f)(x+\varepsilon)-(E_{[n]} f)(x-\varepsilon),
\]
\[
b^{(j)}_{f}(x)\coloneqq E_{[n]\setminus\{j\}}f(x+e_{j})-E_{[n]\setminus \{j\}}f(x-e_{j})=2D_{j}E_{[n]\setminus\{j\}}(f)(x).
\]
Naor and Schechtman \cite{N-S2016} conjectured that there exists a constant $c_{p}>0$ such that for any $n$, $m\in \mathbb{N}$ the following inequality holds for every $f:\Z\to \mathbb{R}$
\begin{equation}\label{concolution inequality of N-S}
\|H_{f}\|_{L_{p}(\Z\times\Omega_{n})}\leq c_{p}\left(\left\|\sum_{j=1}^{n}\varepsilon_{j}b^{(j)}_{f}\right\|_{L_{p}(\Z\times \Omega_{n})}+\sum_{j=1}^{n}\left\|(T_{j}-\mathrm{Id})f\right\|_{L_{p}(\Z\times \Omega_{n})}\right).
\end{equation}

In this paper, we settle the problem of Naor and Schechtman \cite[Question 6.1]{N-S2016} by establishing the following strong convolution inequality.

\begin{thm}\label{NS-main}
For every $1\leq p<\infty$, there exists $c_{p}>0$ such that for every $n$, $m\in \N$, the following inequality holds for every $f:\Z\to \mathbb{R}$
\begin{equation}
\left\|H_{f}\right\|_{L_{p}(\Z\times \Omega_{n})}\leq c_{p}\left\|\sum_{j=1}^{n}\varepsilon_{j}b^{(j)}_{f}\right\|_{L_{p}(\Z\times \Omega_{n})}.
\end{equation}
Moreover, $c_{p}=p-1$ for $p\geq 2$.
\end{thm}

For the endpoint $p=\infty$, we show in Section \ref{endpoint section} that the strong convolution inequality fails at $L_{\infty}$, and that the appropriate endpoint space is the martingale $BMO$ space.
\begin{thm}\label{NS BMO endpoint}
For every $f:\Z\to\mathbb{R}$, the following holds
\begin{equation}\label{endpoint BMO inequality}
\|H_{f}\|_{\mathrm{bmo}_{2}(\mathcal{F})}\leq\left\|\sum_{j=1}^{n}\varepsilon_{j}b^{(j)}_{f}\right\|_{\mathrm{bmo}_{2}(\mathcal{F})}=\left\|\left(\sum_{j=1}^{n}\left|b^{(j)}_{f}\right|^{2}\right)^{1/2}\right\|_{L_{\infty}(\Z)}.
\end{equation}
\end{thm}

The rest of the paper is organized as follows. In Section \ref{main section}, we establish a sequence of decomposition lemmas relating $H_{f}$ to $\{b^{(j)}_{f}\}_{j=1}^{n}$ for a given function $f:\mathbb{Z}^{n}_{m}\to \mathbb{R}$. Using these lemmas, we then apply Burkholder's sharp inequality for martingale transforms to prove Theorem \ref{NS-main-new}, which corresponds to Theorem \ref{NS-main} with $1<p<\infty$. Section \ref{endpoint section} is devoted to the endpoint case of the strong convolution inequality. More precisely, we use the probabilistic representation of the heat semigroup on the hypercube due to Ivanisvili, van Handel, and Volberg \cite{IvHV2020} to derive a gradient estimate. Combining this gradient estimate with a duality argument yields the strong convolution inequality for $p=1$. We also demonstrate the failure of the inequality at $L_{\infty}$. Finally, we conclude with a proof of Theorem \ref{NS BMO endpoint}.

\section{The strong convolution inequality for $1<p<\infty$}\label{main section}
This section is devoted to establishing Theorem \ref{NS-main} for the case $1<p<\infty$ via martingale theory. We begin with introducing notations that will be used in the sequel. For each $A\subseteq [n]$ and $\varepsilon\in \Omega_{n}$,  let $S^{\varepsilon}_{A}\coloneqq \frac{T^{\varepsilon}_{A}+T^{-\varepsilon}_{A}}{2}$.
For a given $f:\Z\to \mathbb{R}$, we define
\[
\mathcal{S}f(x,\varepsilon)\coloneqq \frac{f(x+\varepsilon)+f(x-\varepsilon)}{2}=\frac{T^{\varepsilon}_{[n]}f(x)+T^{-\varepsilon}_{[n]}f(x)}{2}.
\]
For each $j\in[n]$, we let $\varepsilon_{j}:\{-1,1\}^{n}\to \{-1,1\}$ be the $j$-th coordinate function. Let $\mathcal{B}$ be the Borel $\sigma$-algebra on $\Z$. We define a sequence of $\sigma$-algebras on $\Z\times \Omega_{n}$ by setting $\mathcal{F}_0=\{B\times \Omega_{n}: B\in\mathcal{B}\}$ and
\[
\mathcal{F}_{j}\coloneqq \sigma\{B\times A: B\in \mathcal{B}, A\subseteq \Omega_{j}\}\qquad  j\in [n].
\]
Then, $\mathcal{F}_{0}\subseteq \mathcal{F}_{1}\subseteq \cdots\subseteq \mathcal{F}_{n}$ forms a filtration on the probability space $\Z\times \Omega_{n}$.

The following lemma provides a representation of the conditional expectation of $\mathcal{S}(b^{(j)}_{f})$.
\begin{lem}\label{representation 1}
For every $f:\Z\to \mathbb{R}$ and  each $j\in [n]$, the following identity holds
\[
\begin{split}
\mathbb{E}\left[\mathcal{S}b^{(j)}_{f}\Big{|}\mathcal{F}_{j-1}\right](x,\varepsilon)=E_{[n]\setminus[j-1]}S^{\varepsilon}_{[j-1]}b^{(j)}_{f}(x),\quad (x,\varepsilon)\in \Z\times \Omega_{n}.
\end{split}
\]
\end{lem}
\begin{proof}
By the definition of $\mathcal{F}_{j-1}$, we have 
\begin{equation}\label{expansion 1}
\begin{split}
&\mathbb{E}\left[\mathcal{S}b^{(j)}_{f}\Big{|}\mathcal{F}_{j-1}\right](x,\varepsilon)\\
=&\frac{1}{2^{n-j+1}}\sum_{\varepsilon^{\prime}\in \Omega_{n-j+1}}\frac{b^{(j)}_{f}(x+\sum_{l=1}^{j-1}\varepsilon_{l}e_{l}+\sum_{l=j}^{n}\varepsilon^{\prime}_{l-j-1}e_{l})+b^{(j)}_{f}(x-\sum_{l=1}^{j-1}\varepsilon_{l}e_{l}-\sum_{l=j}^{n}\varepsilon^{\prime}_{l-j-1}e_{l})}{2}\\
=&\frac{1}{2^{n-j+1}}\sum_{\varepsilon^{\prime}\in \Omega_{n-j+1}}\frac{\left(T^{(\varepsilon_{1},\cdots, \varepsilon_{j-1},\varepsilon^{\prime}_{1},\cdots, \varepsilon^{\prime}_{n-j+1})}_{[n]}b^{(j)}_{f}\right)(x)+\left(T^{(-\varepsilon_{1},\cdots, -\varepsilon_{j-1},-\varepsilon^{\prime}_{1},\cdots, -\varepsilon^{\prime}_{n-j+1})}_{[n]}b^{(j)}_{f}\right)(x)}{2}\\
=&D_{j}E_{[n]\setminus\{j\}}\left(\frac{1}{2^{n-j+1}}\sum_{\varepsilon^{\prime}\in \Omega_{n-j+1}}\left(T^{(\varepsilon_{1},\cdots, \varepsilon_{j-1},\varepsilon^{\prime}_{1},\cdots, \varepsilon^{\prime}_{n-j+1})}_{[n]}b^{(j)}_{f}\right)(x)\right)\\
&+D_{j}E_{[n]\setminus\{j\}}\left(\frac{1}{2^{n-j+1}}\sum_{\varepsilon^{\prime}\in \Omega_{n-j+1}}\left(T^{(-\varepsilon_{1},\cdots, -\varepsilon_{j-1},-\varepsilon^{\prime}_{1},\cdots, -\varepsilon^{\prime}_{n-j+1})}_{[n]}b^{(j)}_{f}\right)(x)\right).
\end{split}
\end{equation}
It is clear that
\begin{equation}\label{expansion 2}
\frac{1}{2^{n-j+1}}\sum_{\varepsilon^{\prime}\in \Omega_{n-j+1}}\left(T^{(\varepsilon_{1},\cdots, \varepsilon_{j-1},\varepsilon^{\prime}_{1},\cdots, \varepsilon^{\prime}_{n-j+1})}_{[n]}b^{(j)}_{f}\right)(x)=T^{\varepsilon}_{[j-1]}E_{[n]\setminus[j-1]}f(x),
\end{equation}
and
\begin{equation}\label{expansion 3}
\frac{1}{2^{n-j+1}}\sum_{\varepsilon^{\prime}\in \Omega_{n-j+1}}\left(T^{(-\varepsilon_{1},\cdots, -\varepsilon_{j-1},-\varepsilon^{\prime}_{1},\cdots, -\varepsilon^{\prime}_{n-j+1})}_{[n]}b^{(j)}_{f}\right)(x)=T^{-\varepsilon}_{[j-1]}E_{[n]\setminus[j-1]}f(x).
\end{equation}
Substituting \eqref{expansion 2} and \eqref{expansion 3} into \eqref{expansion 1} yields that
\[
\mathbb{E}\left[\mathcal{S}b^{(j)}_{f}\Big{|}\mathcal{F}_{j-1}\right](x,\varepsilon)=E_{[n]\setminus[j-1]}S^{\varepsilon}_{[j-1]}b^{(j)}_{f}(x),
\]
where we used the fact that all operators are commute with each other and the fact that $b^{(j)}_{f}=2D_{j}E_{[n]\setminus\{j\}}f$.
\end{proof}

Applying Lemma \ref{representation 1}, we now provide the following key decomposition of $H_{f}$ in terms of random combination of $\{\mathcal{S}(b^{(j)}_{f})\}_{j=1}^{n}$.
\begin{lem}\label{main decomposition}
For every $f:\Z\to \mathbb{R}$, we have
\begin{equation}\label{decomposition of Hf}
H_{f}(x,\varepsilon)=\sum_{j=1}^{n}\varepsilon_{j}\mathbb{E}\left[\mathcal{S}b^{(j)}_{f}\Big{|}\mathcal{F}_{j-1}\right](x,\varepsilon), \quad (x,\varepsilon)\in \Z\times \Omega_{n}.
\end{equation}
\end{lem}
\begin{proof}
Note here that $H_{f}=E_{[n]}\left(T^{\varepsilon}_{[n]}-T^{-\varepsilon}_{[n]}\right)f$. For each $j\in [n]$, we define
\begin{equation}\label{decomposition 1}
\mathcal{M}_{j}(f)\coloneqq E_{[n]}E_{[n]\setminus[j]}\left(T^{\varepsilon}_{[j]}-T^{-\varepsilon}_{[j]}\right)(f),
\end{equation}
and let $\mathcal{M}_{0}(f)\equiv 0$ for convenience. Since $E_{j}=\frac{T_{j}+T^{-1}_{j}}{2}$ and $D_{j}=\frac{T_{j}-T^{-1}_{j}}{2}$, it is clear that
\begin{equation}\label{key 1}
T_{j}=E_{j}+D_{j}\qquad\mbox{and}\qquad T^{-1}_{j}=E_{j}-D_{j}.
\end{equation}
By $T^{\varepsilon}_{[j]}=T^{\varepsilon_{j}}_{j}T^{\varepsilon}_{[j-1]}$ and $T^{-\varepsilon}_{[j]}=T^{-\varepsilon_{j}}_{j}T^{-\varepsilon}_{[j-1]}$, it follows from \eqref{key 1} that
\begin{equation}\label{decomposition 2}
T^{\varepsilon}_{[j]}-T^{-\varepsilon}_{[j]}=(E_{j}+\varepsilon_{j}D_{j})T^{\varepsilon}_{[j-1]}-(E_{j}-\varepsilon_{j}D_{j})T^{-\varepsilon}_{[j-1]}.
\end{equation}
Substituting \eqref{decomposition 2} into \eqref{decomposition 1} yields that
\begin{equation}\label{decomposition 3}
\begin{split}
\mathcal{M}_{j}(f)&=E_{[n]}E_{[n]\setminus[j]}E_{j}\left(T^{\varepsilon}_{[j-1]}-T^{-\varepsilon}_{[j-1]}\right)(f)+2\varepsilon_{j}E_{[n]}E_{[n]\setminus[j]}D_{j}S^{\varepsilon}_{[j-1]}(f)\\
&=\mathcal{M}_{j-1}(f)+2\varepsilon_{j}E_{[n]}E_{[n]\setminus[j]}D_{j}S^{\varepsilon}_{[j-1]}(f)\\
&=\mathcal{M}_{j-1}(f)+\varepsilon_{j}E_{[n]\setminus[j-1]}S^{\varepsilon}_{[j-1]}b^{(j)}_{f}\\
&=\mathcal{M}_{j-1}(f)+\varepsilon_{j}\mathbb{E}\left[\mathcal{S}b^{(j)}_{f}|\mathcal{F}_{j-1}\right],
\end{split}
\end{equation}
where the last equality is due to Lemma  \ref{representation 1}. Noting $M_{n}(f)=H_{f}$,  by induction of \eqref{decomposition 3},  we obtain
\[
H_{f}=\sum_{j=1}^{n}\varepsilon_{j} \mathbb{E}\left[\mathcal{S}b^{(j)}_{f}\Big{|}\mathcal{F}_{j-1}\right].
\]
\end{proof}
The next elementary lemma provides a connection between the martingale difference operator and Lemma \ref{main decomposition}.
\begin{lem}\label{representation of martingale difference}
For each $f:\Z\times \Omega_{n}\to \mathbb{R}$, we have
\[
d_{j}(f)\coloneqq\mathbb{E}\left[f\Big{|}\mathcal{F}_{j}\right]-\mathbb{E}\left[f\Big{|}\mathcal{F}_{j-1}\right]=\varepsilon_{j}\mathbb{E}\left[\varepsilon_{j}f\Big{|}\mathcal{F}_{j-1}\right],\quad1\leq j\leq n.
\]
\end{lem}
\begin{proof}
The proof follows from the Walsh expansion directly. Indeed, for $f:\Z\times \Omega_{n}\to \mathbb{R}$, we define $F:\Z\to L_{2}(\Omega_{n};\mathbb{R})$ by
\[
F(x)\coloneqq f(x,\cdot),\quad x\in \Z.
\]
Since $\{w_{A}\}_{A\subseteq [n]}$ forms an orthogonal basis of $L_{2}(\Omega_{n})$, it follows that for each $x\in \Z$,  we have
\[
F(x)=\sum_{A\subseteq [n]}\widehat{F(x)}(A)w_{A}.
\]
By the definition of $\{\mathcal{F}_{j}\}_{j=0}^{n}$, we have
\[
\mathbb{E}[F|\mathcal{F}_{j}](x)=\sum_{A\subseteq [j]}\widehat{F(x)}(A)w_{A}.
\]
Thus,
\begin{equation}\label{Walsh expansion 1}
\begin{split}
d_{j}(f)(x,\cdot)&=\sum_{\substack{A\subseteq [n]\\\max(A)=j}}\widehat{F(x)}(A)w_{A}\\
&=\sum_{B\subseteq [j-1]}\widehat{F(x)}(B\cup\{j\})w_{B\cup\{j\}}=\varepsilon_{j}\cdot\left(\sum_{B\subseteq [j-1]}\widehat{F(x)}(B\cup\{j\})w_{B}\right).
\end{split}
\end{equation}
On th other hand, by the definition of conditional expectation, we have
\begin{equation}\label{Walsh expansion 2}
\begin{split}
\mathbb{E}\left[\varepsilon_{j} F\Big{|}\mathcal{F}_{j-1}\right](x,\cdot)&=\mathbb{E}\left[\sum_{j\in A\subseteq [n]}\widehat{F(x)}(A)w_{A\setminus\{j\}}\Big{|}\mathcal{F}_{j-1}\right]\\
&\quad +\mathbb{E}\left[\sum_{A\subseteq [n]\setminus\{j\}}\widehat{F(x)}(A)w_{A\cup\{j\}}\Big{|}\mathcal{F}_{j-1}\right]\\
&=\sum_{\substack{A\subseteq[n]\\\max(A)=j}}\widehat{F(x)}(A)w_{A\setminus\{j\}}=\sum_{B\subseteq [j-1]}\widehat{F(x)}(B\cup\{j\})w_{B}.
\end{split}
\end{equation}
Combining \eqref{Walsh expansion 1} with \eqref{Walsh expansion 2} yields the desired result.
\end{proof}

We now recall the \emph{sharp} inequality for martingale transform of Burkholder \cite{Bur1984}.
\begin{thm}[Burkholder]\label{Burkholder martingale transform}
Suppose that $M=\{M_{j}\}_{j=0}^{\infty}$ martingale with $M_{0}\equiv 0$. Let $d_{j}\coloneqq M_{j}-M_{j-1}$ be the martingale difference for each $j\in \mathbb{N}$. If $\{\varepsilon_{j}\}_{j=1}^{n}$ is a sequence of numbers in $\{-1,1\}$, then the following inequality holds
\[
\left\|\sum_{j=1}^{\infty}\varepsilon_{j}d_{j}\right\|_{L_{p}(\Omega)}\leq (p^{*}-1)\|M\|_{L_{p}(\Omega)},  \quad 1<p<\infty,
\]
where $p^{*}=\max\{p,p/(p-1)\}$
\end{thm}

The strong convolution inequality for $L_{p}$ with $1<p<\infty$ is stated and proved as follows.
\begin{thm}\label{NS-main-new}
For every $1<p<\infty$, there exists $c_{p}>0$ such that for every $n$, $m\in \N$, the following inequality holds for every $f:\Z\to \mathbb{R}$
\begin{equation}
\left\|H_{f}\right\|_{L_{p}(\Z\times \Omega_{n})}\leq (p^{*}-1)\left\|\sum_{j=1}^{n}\varepsilon_{j}b^{(j)}_{f}\right\|_{L_{p}(\Z\times \Omega_{n})},
\end{equation}
where $p^{*}=\max\{p,\frac{p}{p-1}\}$.
\end{thm}

\begin{proof}
By Lemma \ref{main decomposition}, we have
\[
H_{f}(x,\varepsilon)=\sum_{j=1}^{n}\varepsilon_{j}\mathbb{E}\left[\mathcal{S}b^{(j)}_{f}|\mathcal{F}_{j-1}\right](x,\varepsilon).
\]
Hence, for each $g\in L_{p^{\prime}}(\Z\times \Omega_{n})$ with $\|g\|_{L_{p^{\prime}}}\leq 1$ and $p^{\prime}=\frac{p}{p-1}$, we have
\begin{equation}\label{martingale transform 1}
\begin{split}
\mathbb{E}_{\nu\otimes \mu}\left[gH_{f}\right]&=\mathbb{E}_{\nu\otimes \mu}\left[\left(\sum_{j=1}^{n}\mathbb{E}[\varepsilon_{j} g|\mathcal{F}_{j-1}]\right)\cdot\mathcal{S}b^{(j)}_{f}\right]\\
&=\mathbb{E}_{\nu\otimes \mu}\left[\sum_{j=1}^{n}d_{j}(g)\cdot \varepsilon_{j} \mathcal{S}b^{(j)}_{f}\right],
\end{split}
\end{equation}
where we used Lemma \ref{representation of martingale difference}.

Let $(\delta^{\prime}_{1},\cdots, \delta^{\prime}_{n})$ be an i.i.d. Rademacher sequence (on a probability space $(\Omega,\mathcal{B},\mu^{\prime})$), which is independent of $(\varepsilon_{1},\cdots,\varepsilon_{n})$. It is clear that for each $\{\alpha_{j}\}_{j=1}^{n}$ and $\{\beta_{j}\}_{j=1}^{n}\subseteq \mathbb{R}$ we have
\begin{equation}\label{martingale transform 2}
\sum_{j=1}^{n}\alpha_{j}\beta_{j}=\mathbb{E}_{\mu^{\prime}}\left[\left(\sum_{j=1}^{n}\delta_{j}\alpha_{j}\right)\cdot\left(\sum_{j=1}^{n}\delta_{j}\beta_{j}\right)\right].
\end{equation}
Substituting \eqref{martingale transform 2} into \eqref{martingale transform 1}  yields
\begin{equation}\label{martingale transform 3}
\begin{split}
|\mathbb{E}_{\nu\otimes \mu}\left[gH_{f}\right]|&=\mathbb{E}_{\nu\otimes \mu\otimes \mu^{\prime}}\left[\left(\sum_{j=1}^{n}\delta_{j}d_{j}(g)\right)\cdot \left(\sum_{j=1}^{n}\delta_{j}\varepsilon_{j} \mathcal{S}b^{(j)}_{f}\right)\right]\\
&\leq \left\|\sum_{j=1}^{n}\delta_{j}d_{j}(g)\right\|_{L_{p^{\prime}}(\Z\times \Omega_{n}\times \Omega)}\cdot\left\|\sum_{j=1}^{n}\delta_{j}\varepsilon_{j} \mathcal{S}b^{(j)}_{f}\right\|_{L_{p}(\Z\times \Omega_{n}\times \Omega)}
\end{split}
\end{equation}
Note that for each $j\in [n]$ the random variable $\delta_{j}\varepsilon_{j}$ has the same distribution of $\varepsilon_{j}$, then it follows that
\begin{equation}\label{martingale transform 4}
\left\|\sum_{j=1}^{n}\delta_{j}\varepsilon_{j} \mathcal{S}b^{(j)}_{f}\right\|_{L_{p}(\Z\times \Omega_{n}\times \Omega)}=\left\|\sum_{j=1}^{n}\varepsilon_{j}\mathcal{S}b^{(j)}_{f}\right\|_{L_{p}(\Z\times \Omega_{n})}\leq \left\|\sum_{j=1}^{n}\varepsilon_{j}b^{(j)}_{f}\right\|_{L_{p}(\Z\times \Omega_{n})},
\end{equation}
where we used the fact that $\mathcal{S}$ is a contraction.

We now apply Theorem \ref{Burkholder martingale transform} to bound $\left\|\sum_{j=1}^{n}\delta_{j}d_{j}(g)\right\|_{L_{p^{\prime}}(\Z\times \Omega_{n}\times \Omega)}$. For each $\omega\in \Omega$, by Theorem \ref{Burkholder martingale transform}, we have
\[
\left\|\sum_{j=1}^{n}\delta_{j}(\omega)d_{j}(g)\right\|_{L_{p^{\prime}}(\Z\times \Omega_{n})}\leq (p^{*}-1)\|g\|_{L_{p^{\prime}}(\Z\times \Omega_{n})}.
\]
By the Fubini theorem, it yields that
\begin{equation}\label{martingale transform 5}
\begin{split}
\left\|\sum_{j=1}^{n}\delta_{j}d_{j}(g)\right\|_{L_{p^{\prime}}(\Z\times \Omega_{n}\times \Omega)}&=\left(\mathbb{E}_{\mu^{\prime}}\left[\left\|\sum_{j=1}^{n}\delta_{j}(\omega)d_{j}(g)\right\|^{p^{\prime}}_{L_{p^{\prime}(\Z\times \Omega_{n})}}\right]^{p^{\prime}}\right)^{1/p^{\prime}}\\
&\leq (p^{*}-1)\|g\|_{L_{p^{\prime}}(\Z\times \Omega_{n})}\leq p^{*}-1.
\end{split}
\end{equation}
Combining \eqref{martingale transform 3}, \eqref{martingale transform 4}, and \eqref{martingale transform 5} together yields
\[
\|H_{f}\|_{L_{p}(\Z\times \Omega_{n})}=\sup_{\|g\|_{L_{p^{\prime}}(\Z\times \Omega_{n})}\leq 1}|\mathbb{E}_{\nu\otimes \mu}[gH_{f}]|\leq (p^{*}-1)\left\|\sum_{j=1}^{n}\varepsilon_{j}b^{(j)}_{f}\right\|_{L_{p}(\Z\times \Omega_{n})},
\]
which completes our proof.
\end{proof}

\begin{rem}
To bound the term $\left\|\sum_{j=1}^{n}\delta_{j}d_{j}(g)\right\|_{L_{p^{\prime}}(\Z\times \Omega_{n}\times \Omega)}$, we may use the Burkholder-Gundy inequality directly. However, the constant $c_{p}$ we obtain in this way will be worse than the sharp inequality of Burkholder as we shown as above.
\end{rem}

\section{Endpoint estimates for the strong convolution inequality}\label{endpoint section}

In this section, we establish the endpoint case of the strong convolution inequality. More precisely, we show that the inequality holds for $L_{1}$ and fails for $L_{\infty}$. The appropriate endpoint for $p=\infty$ is the martingale BMO space.

\subsection{The validity for $L_{1}$ and the failure for the $L_{\infty}$}
One of the key ingredient to establish the $L_{1}$ estimate is the following probabilistic representation of heat semigroup on the hypercube, which is due to Ivanisvili, van Handel, and Volberg \cite{IvHV2020}. Given $t\geq 0$, let $\xi(t)=\{\xi_j(t)\}_{j=1}^n$ be a random vector taking values in the hypercube $\{-1,1\}^{n}$, whose coordinates $\xi_{j}(t)$ are independent and identically distributed with
\begin{equation}\label{xi}
\mathbb{P}\{\xi_{j}(t)=1\}=\frac{1+e^{-t}}{2},\quad \mathbb{P} \{\xi_{j}(t)=-1\}=\frac{1-e^{-t}}{2}.
\end{equation}
The normalized $\xi(t)$, denoted by $\delta(t)=\{\delta_{j}(t)\}_{j=1}^n$, satisfies (see also \cite[p. 668]{IvHV2020})
\begin{equation}\label{delta}
	\delta_{j}(t)=\frac{\xi_{j}(t)-e^{-t}}{\sqrt{1-e^{-2t}}},\quad j\in [n].
\end{equation}
For each $j\in [n]$, the $j$-th partial derivative is defined by
\[
\partial_{j}f(\varepsilon)\coloneqq\frac{f(\varepsilon)-f(\varepsilon^{\oplus_{j}})}{2},
\]
where $\varepsilon^{\oplus_{j}}\coloneqq (\varepsilon_{1},\cdots,\varepsilon_{j-1},-\varepsilon_{j},\varepsilon_{j+1},\cdots, \varepsilon_{n})$.

\begin{prop}[Ivanisvili-van Handel-Volberg]\label{IvHv}
Suppose that $B$ is a Banach space and $\{\xi_{j}(t)\}_{j=1}^{n}$, $\{\delta_{j}(t)\}_{j=1}^{n}$ are random variables as above for each $t\geq 0$. We have the following identities:
\[
e^{-t\Delta}(f)(\varepsilon)=\mathbb{E}_{\xi(t)}\left[f\left(\varepsilon_{1}\xi_{1}(t),\cdots,\varepsilon_{n}\xi_{n}(t)\right)\right]
\]
and
\[
\partial_{j}e^{-t\Delta}(f)(\varepsilon)=\frac{1}{\sqrt{e^{2t}-1}}\mathbb{E}_{\xi(t)}\left[\delta_{j}(t)f\left(\varepsilon_{1}\xi_{1}(t),\cdots,\varepsilon_{n}\xi_{n}(t)\right)\right]
\]
for all $\varepsilon\in \{-1,1\}^{n}$, $j\in [n]$ and $f: \{-1,1\}^{n}\to B$.
\end{prop}
The following inequality may be known for experts, but we include the proof here for self-contained.
\begin{lem}\label{endpoint cube estimate}
For every $g:\Omega_{n}\to\mathbb{R}$, the following inequality holds
\begin{equation}\label{endpoint inverse gradient}
\left\|\left(\sum_{j=1}^{n}\left|\partial_{j}\Delta^{-1}g\right|^{2}\right)^{1/2}\right\|_{L_{\infty}(\Omega_{n})}
\leq\frac{\pi}{2}\|g\|_{L_{\infty}(\Omega_{n})}.
\end{equation}
\end{lem}
\begin{proof}
By Proposition \ref{IvHv}, we have
\begin{equation}\label{orthonormal 1}
\sum_{j=1}^{n}\left|\partial_{j}e^{-t\Delta}(g)(\varepsilon)\right|^{2}=\frac{1}{e^{2t}-1}\sum_{j=1}^{n}\mathbb{E}_{\xi(t)}\left|\left[\delta_{j}(t)g\left(\varepsilon_{1}\xi_{1}(t),\cdots,\varepsilon_{n}\xi_{n}(t)\right)\right]\right|^{2}.
\end{equation}
Note that $\{\delta_{j}\}_{j=1}^{n}$ forms an orthonormal sequence, hence \eqref{orthonormal 1} yields that
\begin{equation}\label{orthonormal 2}
\sum_{j=1}^{n}\left|\partial_{j}e^{-t\Delta}(g)(\varepsilon)\right|^{2}\leq \frac{1}{e^{2t}-1}\|g\|^{2}_{L_{2}(\Omega_{n},\mu_{\xi(t)})}\leq \frac{1}{e^{2t}-1}\|g\|^{2}_{L_{\infty}(\Omega_{n})}.
\end{equation}
By the spectral theorem of $\Delta$, we have
\begin{equation}\label{spectrum decomposition}
\Delta^{-1}=\int_{0}^{\infty}e^{-t\Delta}~dt.
\end{equation}
Combining \eqref{orthonormal 2} with \eqref{spectrum decomposition} yields that
\begin{align*}
\left(\sum_{j=1}^{n}\left|\partial_{j}\Delta^{-1}(g)(\varepsilon)\right|^{2}\right)^{1/2}&\leq \left(\int_{0}^{\infty}\frac{1}{\sqrt{e^{2t}-1}}~dt\right)\|g\|_{L_{\infty}(\Omega_{n})}\\
&=\frac{\pi}{2}\|g\|_{L_{\infty}}.
\end{align*}
This completes our proof.
\end{proof}
Combing Lemma \ref{endpoint cube estimate} with a duality argument yield the following strong convolution inequality for $L_{1}$.
\begin{thm}\label{NS L1 endpoint}
For every $n$, $m\in\mathbb{N}$ and every $f:\Z\to\mathbb{R}$,
\begin{equation}\label{endpoint L1 inequality}
\|H_{f}\|_{L_{1}(\Z\times\Omega_{n})}\leq\frac{\pi}{2}\left\|\left(\sum_{j=1}^{n}\left|b^{(j)}_{f}\right|^{2}\right)^{1/2}\right\|_{L_{1}(\Z)}\leq\frac{\pi}{\sqrt{2}}\left\|\sum_{j=1}^{n}\varepsilon_{j}b^{(j)}_{f}\right\|_{L_{1}(\Z\times\Omega_{n})}.
\end{equation}
\end{thm}
\begin{proof}
To prove the first inequality, we need the following key identities. For any $x\in \Z$ and $\varepsilon\in \Omega_{n}$, we have
\begin{equation}\label{Hf decomposition}
\begin{split}
\partial_{j}H_{f}(x,\varepsilon)&=\frac{E_{[n]}\left(T^{\varepsilon}_{[n]}-T^{-\varepsilon}_{[n]}\right)(f)(x)-E_{[n]}\left(T^{\varepsilon_{\oplus_{j}}}_{[n]}-T^{-\varepsilon_{\oplus_{j}}}_{[n]}\right)(f)(x)}{2}\\
&=\frac{1}{2}E_{[n]}\left(T^{\varepsilon}_{[n]\setminus\{j\}}\left(T^{\varepsilon_{j}}_{j}-T^{-\varepsilon_{j}}_{j}\right)(f)(x)+T^{-\varepsilon}_{[n]\setminus\{j\}}\left(T^{\varepsilon_{j}}_{j}-T^{-\varepsilon_{j}}_{j}\right)(f)(x)\right)\\
&=E_{[n]}S^{\varepsilon}_{[n]\setminus{j}}(T^{\varepsilon_{j}}_{j}-T^{-\varepsilon_{j}}_{j})(f)(x)\\
&=2\varepsilon_{j}D_{j}E_{[n]}S^{\varepsilon}_{[n]\setminus\{j\}}(f)(x),
\end{split}
\end{equation}
where we used $\varepsilon_{j}D_{j}=\frac{T^{\varepsilon_{j}}_{j}-T^{-\varepsilon_{j}}_{j}}{2}$ in the last line.

On the other hand, applying the same argument to $\mathcal{S}b^{(j)}_{f}$ we obtain that
\begin{equation}\label{bjf decomposition}
\partial_{j}\left(\varepsilon_{j}\mathcal{S}b^{(j)}_{f}\right)(x,\varepsilon)=\varepsilon_{j}E_{j}S^{\varepsilon}_{[n]\setminus\{j\}}b^{(j)}_{f}(x)=2\varepsilon_{j}D_{j}E_{[n]}S^{\varepsilon}_{[n]\setminus\{j\}}f(x).
\end{equation}
Hence, for each $x\in Z$ and $\varepsilon\in \Omega_{n}$, we have
\begin{equation}\label{endpoint divergence identity}
\Delta H_{f}(x,\varepsilon)=\sum_{j=1}^{n}\partial_{j}\left(\varepsilon_{j}\mathcal{S}b^{(j)}_{f}\right)(x,\varepsilon).
\end{equation}
For any fixed $x\in \Z$, it is clear that $H_{f}(x,-\varepsilon)=-H_{f}(x,\varepsilon)$, and the symmetry gives $\mathbb{E}_{\mu}H_{f}(x,\cdot)=0$. Hence, \eqref{endpoint divergence identity} is equivalent to the following
\begin{equation}\label{divergence 2}
H_{f}(x,\varepsilon)=\sum_{j=1}^{n}\Delta^{-1}D_{j}\left(\varepsilon_{j}\mathcal{S}b^{(j)}_{f}\right)(x,\varepsilon),
\end{equation}
for each $(x,\varepsilon)\in \Z\times \Omega_{n}$.

Let $g:\Omega_{n}\to\mathbb{R}$ satisfy $\|g\|_{L_{\infty}(\Omega_{n})}\leq1$. By the Cauchy--Schwarz inequality and \eqref{divergence 2}, we obtain from Lemma \ref{endpoint cube estimate} that
\[
\begin{split}
\left|\mathbb{E}_{\mu}\left[gH_{f}(x,\cdot)\right]\right|
&=\left|\mathbb{E}_{\mu}\left[\sum_{j=1}^{n}D_{j}\Delta^{-1}g\cdot\varepsilon_{j}\mathcal{S}b^{(j)}_{f}(x,\cdot)\right]\right|\\
&\leq\frac{\pi}{2}\mathbb{E}_{\mu}\left[\left(\sum_{j=1}^{n}\left|\mathcal{S}b^{(j)}_{f}(x,\cdot)\right|^{2}\right)^{1/2}\right].
\end{split}
\]
Taking the supremum over $g\in L_{\infty}(\Omega_{n})$ with $\|g\|_{L_{\infty}(\Omega_{n})}\leq 1$, we have
\begin{equation}\label{endpoint pointwise L1}
\begin{split}
\|H_{f}(x,\cdot)\|_{L_{1}(\Omega_{n})}
&\leq\frac{\pi}{4}\mathbb{E}_{\mu}\left[\left(\sum_{j=1}^{n}\left|b^{(j)}_{f}(x+\varepsilon)\right|^{2}\right)^{1/2}
+\left(\sum_{j=1}^{n}\left|b^{(j)}_{f}(x-\varepsilon)\right|^{2}\right)^{1/2}\right]\\
&=\frac{\pi}{2}E_{[n]}\left[\left(\sum_{j=1}^{n}\left|b^{(j)}_{f}\right|^{2}\right)^{1/2}\right](x),
\end{split}
\end{equation}
for every $x\in \Z$. Taking the expectation with respect to $x\in \Z$, \eqref{endpoint pointwise L1} and the contraction of $E_{[n]}$ yield the first inequality in \eqref{endpoint L1 inequality}. For the second inequality, it suffices to use the sharp Khintchine inequality \cite{Sza1976} to get the desired bound.
\end{proof}

We conclude this subsection with the failure of the strong convolution inequality for $L_{\infty}$.

\begin{prop}\label{NS Linfty counterexample}
For every integer $k\geq1$, there is a real function $f$ on $\mathbb{Z}_{3}^{2k+1}$ such that
\begin{equation}\label{endpoint counterexample bounds}
\left\|\sum_{j=1}^{2k+1}\varepsilon_{j}b^{(j)}_{f}\right\|_{L_{\infty}(\mathbb{Z}_{3}^{2k+1}\times\Omega_{2k+1})}=1 \quad\mbox{and}\quad
\|H_{f}\|_{L_{\infty}(\mathbb{Z}_{3}^{2k+1}\times\Omega_{2k+1})}
=\sum_{r=1}^{k}\frac{1}{r}.
\end{equation}
\end{prop}
\begin{proof}
Set $n=2k+1$ and $m=3$. Since $T_{j}^{3}=\mathrm{Id}$, the operator
\[
P_{j}\coloneqq\frac{\mathrm{Id}+T_{j}+T_{j}^{-1}}{3}
\]
is a projection with $D_{j}P_{j}=0$. Hence the following identities hold $E_{j}=\frac{3P_{j}-\mathrm{Id}}{2}$, $E_{j}^{-1}=3P_{j}-2\mathrm{Id}$, and $D_{j}E_{j}^{-1}=-2D_{j}$. For any $u:\mathbb{Z}^{n}_{m}\to \mathbb{R}$ let
\[
f\coloneqq E_{[n]}^{-1}u=\prod_{j=1}^{n}(3P_{j}-2\mathrm{Id})u.
\]
Then we have
\begin{equation}\label{endpoint ternary reduction}
H_{f}(x,\varepsilon)=u(x+\varepsilon)-u(x-\varepsilon)\quad\mbox{and}\quad
b^{(j)}_{f}(x)=-2\left(u(x+e_{j})-u(x-e_{j})\right).
\end{equation}

\noindent For $t\in\{0,\cdots,n\}$, we define the function
\[
\varphi(t)\coloneqq
\begin{cases}
0,&0\leq t\leq k+1,\\[2pt]
\displaystyle\frac{1}{2}\sum_{r=n-t+1}^{k}\frac{1}{r},&k+2\leq t\leq n.
\end{cases}
\]
Thus, for $0\leq t<n$,
\begin{equation}\label{endpoint harmonic increments}
\varphi(t+1)-\varphi(t)=
\begin{cases}
0,&0\leq t\leq k,\\
\displaystyle\frac{1}{2(n-t)},&k+1\leq t<n.
\end{cases}
\end{equation}
Represent the elements of $\mathbb{Z}_{3}$ by $\bar{0}$, $\bar{1}$, and $\bar{2}$, and put
\[
u(x)\coloneqq\varphi\left(|\{j\in [n]:x_{j}=\bar{1}\}|\right)-\varphi\left(|\{j\in [n]:x_{j}=\bar{2} \}|\right).
\]
Fix $x\in \Z$ and write $a=|\{j\in[n]:x_{j}=\bar{1} \}|$ and $b=|\{j\in [n]:x_{j}=\bar{2} \}|$. A direct computation gives
\begin{equation}\label{endpoint coordinate differences}
u(x+e_{j})-u(x-e_{j})=
\begin{cases}
\varphi(a+1)-\varphi(a)+\varphi(b+1)-\varphi(b),&x_{j}=\bar{0},\\
\varphi(b)-\varphi(b+1),&x_{j}=\bar{1},\\
\varphi(a)-\varphi(a+1),&x_{j}=\bar{2}.
\end{cases}
\end{equation}

If $a\leq k$ and $b\leq k$, all differences in \eqref{endpoint coordinate differences} vanish. If $a\geq k+1$, by $a+b\leq2k+1$, $b\leq k$. The coordinates with $x_{j}=\bar{1}$ give zero; if $a<n$, each of the remaining $n-a$ coordinates gives absolute difference $1/(2(n-a))$. If $a=n$, every difference is zero. Interchanging $a$ and $b$ proves the remaining case. By \eqref{endpoint ternary reduction}, this establishes
\[
\sum_{j=1}^{n}\left|b^{(j)}_{f}(x)\right|\leq1
\quad\mbox{for every }x\in\mathbb{Z}_{3}^{n}.
\]
Equality holds, for example, when $a=k+1$ and $b=k$. Applying the fact
\[
\left\|\sum_{j=1}^{n}\varepsilon_{j}\alpha_{j}\right\|_{L_{\infty}(\Omega_{n})}=\sum_{j=1}^{n}|\alpha_{j}|,
\]
for all $\{\alpha_{j}\}_{j=1}^{n}\subseteq \mathbb{R}$, we derive the first equality in \eqref{endpoint counterexample bounds}.

Let $\mathbf{1}=(1,\cdots,1)\in\Omega_{n}$. For $x=0$ and $\varepsilon=\mathbf{1}$, the first equality in \eqref{endpoint ternary reduction} turns into
\[
H_{f}(0,\mathbf{1})=2\varphi(n)=\sum_{r=1}^{k}\frac{1}{r}.
\]
Since $-\varphi(n)\leq u(x)\leq\varphi(n)$ for every $x\in \mathbb{Z}^{n}_{m}$, it entails that
\[
\|H_{f}\|_{L_{\infty}(\mathbb{Z}^{n}_{m}\times \Omega_{n})}=2\varphi(n).
\]
This proves the second equality in \eqref{endpoint counterexample bounds}. Finally,
\[
\sum_{r=1}^{k}\frac{1}{r}\geq\int_{1}^{k+1}\frac{dt}{t}=\log(k+1)\longrightarrow\infty,
\]
which proves the asserted failure.
\end{proof}

\subsection{The martingale BMO estimate}

Let $\mathcal{F}=\{\mathcal{F}_{j}\}_{j=0}^{n}$ be the filtration from Section \ref{main section}. For every $h:\Z\times \Omega_{n}\to \mathbb{R}$ with $\mathbb{E}[h|\mathcal{F}_{0}]=0$, define the BMO norm as follows
\begin{equation}\label{endpoint bmo definition}
\|h\|_{\mathrm{bmo}_{2}(\mathcal{F})}
\coloneqq\max_{0\leq k\leq n}
\left\|\mathbb{E}\left[\left|h-\mathbb{E}[h|\mathcal{F}_{k}]\right|^{2}\Big{|}\mathcal{F}_{k}\right]\right\|_{L_{\infty}(\Z\times\Omega_{n})}^{1/2}.
\end{equation}
We now prove the strong convolution inequality for the BMO endpoint as follows.

\begin{proof}[Proof of Theorem \ref{NS BMO endpoint}]
Since the summands in \eqref{decomposition of Hf} are martingale differences, Lemma \ref{main decomposition} gives
\begin{equation}\label{endpoint H differences}
d_{j}(H_{f})=\varepsilon_{j}\mathbb{E}\left[\mathcal{S}b^{(j)}_{f}\Big{|}\mathcal{F}_{j-1}\right].
\end{equation}
For $k<i<j$, the tower property and $\mathbb{E}[d_{j}(H_{f})|\mathcal{F}_{j-1}]=0$ imply
\begin{equation}\label{conditional orthonormal 1}
\mathbb{E}\left[d_{i}(H_{f})d_{j}(H_{f})\Big{|}\mathcal{F}_{k}\right]=0.
\end{equation}
Therefore, for $0\leq k<n$, conditional orthogonality, we apply \eqref{endpoint H differences} and \eqref{conditional orthonormal 1} to obtain that
\begin{equation}\label{endpoint conditional energy}
\begin{split}
\mathbb{E}\left[\left|H_{f}-\mathbb{E}[H_{f}|\mathcal{F}_{k}]\right|^{2}\Big{|}\mathcal{F}_{k}\right]&=\sum_{j=k+1}^{n}\mathbb{E}\left[\left|\mathbb{E}\left[\mathcal{S}b^{(j)}_{f}\Big{|}\mathcal{F}_{j-1}\right]\right|^{2}\Big{|}\mathcal{F}_{k}\right]\\
&\leq\mathbb{E}\left[\sum_{j=k+1}^{n}\left|\mathcal{S}b^{(j)}_{f}\right|^{2}\Big{|}\mathcal{F}_{k}\right],
\end{split}
\end{equation}
where the last inequality follows from the Jensen inequality. For each $(x,\varepsilon)\in \Z\times \Omega_{n}$, the definition of $\mathcal{S}$ yields
\begin{equation}\label{maximal estimate 1}
\sum_{j=k+1}^{n}\left|\mathcal{S}b^{(j)}_{f}(x,\varepsilon)\right|^{2}\leq\frac{1}{2}\sum_{j=k+1}^{n}\left(\left|b^{(j)}_{f}(x+\varepsilon)\right|^{2}+\left|b^{(j)}_{f}(x-\varepsilon)\right|^{2}\right)\leq\max_{y\in\Z}\sum_{j=1}^{n}\left|b^{(j)}_{f}(y)\right|^{2}.
\end{equation}
Combining \eqref{endpoint conditional energy} with \eqref{maximal estimate 1}, we obtain
\[
\left\|H_{f}\right\|_{\mathrm{bmo}_{2}(\mathcal{F})}\leq \left\|\left(\sum_{j=1}^{n}\left|b^{(j)}_{f}(y)\right|^{2}\right)^{1/2}\right\|_{L_{\infty}(\Z)}.
\]
Furthermore, by the filtration $\mathcal{F}=\{\mathcal{F}_{j}\}_{j=0}^{n}$ as in Section \ref{main section}, the following equalities hold
\[
\mathbb{E}\left[\sum_{j=1}^{n}\varepsilon_{j}b^{(j)}_{f}\Big{|}\mathcal{F}_{k}\right]=\sum_{j=1}^{k}\varepsilon_{j}b^{(j)}_{f}
\]
and
\[
\mathbb{E}\left[\left|\sum_{j=k+1}^{n}\varepsilon_{j}b^{(j)}_{f}\right|^{2}\Big{|}\mathcal{F}_{k}\right]=\sum_{j=k+1}^{n}\left|b^{(j)}_{f}\right|^{2}.
\]
The maximum over $k$ is attained at $k=0$, which establishes the first inequality in \eqref{endpoint BMO inequality}.

To prove sharpness, we take $n=1$, $m\geq5$, and $f(x)=\sin(2\pi x/m)$. Then, the following holds
\[
b^{(1)}_{f}(x)=2\sin(2\pi/m)\cos(2\pi x/m)\quad\mbox{and}\quad
H_{f}(x,\varepsilon)=\varepsilon_{1}\cos(2\pi/m)b^{(1)}_{f}(x).
\]
For this one-step filtration, the $BMO$ norm of $\varepsilon_{1}b^{(1)}_{f}$ equals $\|b^{(1)}_{f}\|_{L_{\infty}}$. Hence the ratios on the two sides of the inequality are $\cos(2\pi/m)$, which tends to $1$ as $m\to\infty$.
\end{proof}


\section*{Acknowledgment}

This work was supported by the National Natural Science Foundation of China (Grant Nos. 12125109, W2411005 \& 12671165); the Natural Science Foundation of Hunan Province (Grant Nos: 2025ZYJ002, 2024JJ1010 \& 2024RC3040); the Scientific Research Fund of Hunan Provincial Education Department (Grant Nos. 25A0009 \& 25B0008).

During preparation of the manuscript, GPT was used only for language editing and preliminary consistency checks. All mathematical ideas, methods, and proofs originated entirely with the authors, who take full responsibility for the manuscript.


\providecommand{\bysame}{\leavevmode\hbox to3em{\hrulefill}\thinspace}
\providecommand{\MR}{\relax\ifhmode\unskip\space\fi MR }
\providecommand{\MRhref}[2]{%
  \href{http://www.ams.org/mathscinet-getitem?mr=#1}{#2}
}
\providecommand{\href}[2]{#2}

\end{document}